\documentclass[11pt,a4paper,twoside,reqno]{amsart}

\usepackage{epsfig,
curves}
\usepackage{float}
\usepackage{amsfonts}
\usepackage{amssymb}

\usepackage{fontenc}
\usepackage{amsmath}
\usepackage{latexsym}

\newcommand{\N}{{\mathbb N}}

\newcommand{\Z}{{\mathbb Z}}

\newcommand{\diam}{{\rm diam}}

\newcommand{\dis}{{\rm dis}}

\newcommand{\vol}{{\rm vol}}
\newcommand{\card}{{\rm card}}

\newcommand{\length}{{\rm length}}

\newcommand{\cf}{{\it cf.}}
\newcommand{\ie}{{\it i.e.}}
\newcommand{\oo}{O}

\numberwithin{equation}{section}

\newtheorem{theorem}{Theorem}[section]
\newtheorem{proposition}[theorem]{Proposition}
\newtheorem{corollary}[theorem]{Corollary}
\newtheorem{lemma}[theorem]{Lemma}

\theoremstyle{definition}
\newtheorem{definition}[theorem]{Definition}
\newtheorem{example}[theorem]{Example}
\newtheorem{remark}[theorem]{Remark}
\newtheorem{notation}[theorem]{Notation}

\long\def\forget#1\forgotten{} %

\title[Growth of quotients of isometry groups]{Growth of quotients of groups acting by isometries on Gromov hyperbolic spaces}

\author[S.~Sabourau]{St\'ephane Sabourau}

\address{Universit\'e Paris-Est,
Laboratoire d'Analyse et Math\'ematiques Appliqu\'ees (UMR 8050), 
UPEC, UPEMLV, CNRS, F-94010, Cr\'eteil, France}

\email{stephane.sabourau@u-pec.fr}



\begin{document}

\begin{abstract}
We show that every non-elementary group~$G$ acting properly and cocompactly by isometries on a proper geodesic Gromov hyperbolic space~$X$ is growth tight.
In other words, the exponential growth rate of~$G$ for the geometric (pseudo)-distance induced by~$X$ is greater than the exponential growth rate of any of its quotients by an infinite normal subgroup.
This result generalizes from a unified framework previous works of Arzhantseva-Lysenok and Sambusetti, and provides an answer to a question of the latter.
\end{abstract}

\maketitle


\forget
Rien n'apparait dans le document !
\forgotten

\section{Introduction}

In this article, we investigate the asymptotic geometry of some discrete groups~$(G,d)$ endowed with a left-invariant metric through their exponential growth rate.
Recall that the exponential growth rate of~$(G,d)$ is defined as 
\begin{equation} \label{eq:om}
\omega(G,d) = \limsup_{R \to +\infty} \frac{\log \card B_G(R)}{R}
\end{equation}
where $B_G(R)$ is the ball of radius~$R$ formed of the elements of~$G$ at distance at most~$R$ from the neutral element~$e$.
The quotient group~$\bar{G}=G/N$ of~$G$ by a normal subgroup~$N$ inherits the quotient distance~$\bar{d}$ given by the least distance between representatives.
The distance~$\bar{d}$ is also left-invariant.
Clearly, we have $\omega(\bar{G},\bar{d}) \leq \omega(G,d)$.
The metric group~$(G,d)$ is said to be \emph{growth tight} if 
\[
\omega(\bar{G},\bar{d}) < \omega(G,d)
\]
for any quotient~$\bar{G}$ of~$G$ by an infinite 
normal subgroup~$N \lhd G$.
In other words, $(G,d)$ is growth tight if it can be characterized by its exponential growth rate among its quotients by an infinite normal subgroup. 
Observe that if the normal subgroup~$N$ is finite, then the exponential growth rates of~$G$ and~$\bar{G}$ clearly agree. \\

The notion of growth tightness was first introduced by R.~Grigorchuk and P.~de~la~Harpe~\cite{grh} for word metrics on finitely generated groups.
In this context, A.~Sambusetti~\cite{sam02b} showed that every nontrivial free product of groups, different from the infinite dihedral group, and every amalgamated product of residually finite groups over finite subgroups are growth tight with respect to any word metric.
In~\cite{al}, G.~Arzhantseva and I.~Lysenok proved that every non-elementary word hyperbolic group in the sense of Gromov is growth tight for any word metric.
This actually gives an affirmative answer to a question of R.~Grigorchuk and P.~de~la~Harpe~\cite{grh}.
On the other hand, it is not difficult to check that the direct product~$F_1 \times F_2$ of two free groups of rank at least~$2$ is not growth tight for the word metric induced by the natural basis of~$F_1 \times F_2$ obtained from two free basis of~$F_1$ and~$F_2$, \cf~\cite{grh}. 

Applications of growth tightness to geometric group theory in connection with the Hopf property and the minimal growth of groups can be found in~\cite{grh,sam01,al,sam02a,sam02b,sam04,css04}. \\

Word metrics are not the only natural metrics which arise on groups.
For instance, let $G$ be the fundamental group of a closed Riemannian manifold~$(M,g)$ with basepoint~$x_0$.
The group~$G$ acts properly and cocompactly by isometries on the Riemannian universal cover~$(\tilde{M},\tilde{g})$ of~$(M,g)$.
The distance on~$G$ induced by~$g$ between two elements~$\alpha, \beta \in G$, denoted by~$d_g(\alpha,\beta)$, is defined as the length of the shortest loop based at~$x_0$ representing $\alpha^{-1} \beta \in G = \pi_1(M,x_0)$.
Every quotient group $\bar{G}=G/N$ by a normal subgroup~$N \lhd G$ is the deck transformation group of the normal cover~$\bar{M}=\tilde{M}/N$ of~$M$.
The quotient distance~$\bar{d_g}$ on~$\bar{G}$ agrees with the distance~$d_{\bar{g}}$ on~$\bar{G}$ induced by the lift~$\bar{g}$ on~$\bar{M}$ of the Riemannian metric~$g$ on~$M$ (here, we take for a basepoint on~$\bar{M}$ any lift of~$x_0$).
Furthermore, the exponential growth rate of~$(\bar{G},d_{\bar{g}})$ agrees with the one of the Riemannian cover~$(\bar{M},\bar{g})$ defined as
\[
\omega(\bar{M},\bar{g}) = \lim_{R \to +\infty} \frac{\log \vol \, B_{\bar{g}}(R)}{R}
\]
where $B_{\bar{g}}(R)$ is the ball of radius~$R$ in~$\bar{M}$ centered at the basepoint (the limit exists since $M$ is compact).
In other words, we have $\omega(\bar{G},d_{\bar{g}}) = \omega(\bar{M},\bar{g})$.
Note that the exponential growth rates of~$\bar{G}$ and~$\bar{M}$ do not depend on the choice of the basepoint.
By definition, the exponential growth rate of the Riemannian universal cover~$(\tilde{M},\tilde{g})$ is the volume entropy of~$(M,g)$. 

In~\cite{sam08}, A.~Sambusetti proved the following Riemannian analogue of G.~Arzhantseva and I.~Lysenok's result~\cite{al}.

\begin{theorem}[\cite{sam08}] \label{theo:sam}
Every Riemannian normal cover~$\bar{M}$ of a closed negatively curved Riemannian manifold~$M$, different from the universal cover~$\tilde{M}$, satisfies $\omega(\bar{M}) < \omega(\tilde{M})$. 
\end{theorem}

As pointed out in~\cite{sam08}, even though the fundamental group of~$M$ is a word hyperbolic group in the sense of Gromov, its geometric distance is only quasi-isometric to any word metric.
Since the exponential growth rate of the fundamental group of~$M$ (hence, {\it apriori}, its growth tightness) is not invariant under quasi-isometries, it is not clear how to derive this theorem from its group-theoretical counterpart.

Clearly, this theorem does not extend to nonpositively curved manifolds: the product of a flat torus with a closed hyperbolic surface provides a simple counterexample.
In~\cite{sam04}, A.~Sambusetti asks if growth tightness holds for any Riemannian metric, without curvature assumption, on a closed negatively curved manifold.
We affirmatively answer this question in a general way unifying different results on the subject, \cf~Theorem~\ref{theo:main}. \\

Given a group~$G$ acting by isometries on a metric space~$X=(X,|\cdot|)$ with origin~$O$, we define a left-invariant pseudo-distance~$d$ on~$G$ as
\begin{equation} \label{eq:dg}
d(\alpha,\beta) = |\alpha(O) \beta(O)|
\end{equation}
for every $\alpha,\beta \in G$, where $|xy|$ represents the distance between a pair of elements~$x,y \in X$.
The notion of exponential growth rate for~$G=(G,d)$ extends to the pseudo-distance~$d$ and does not depend on the choice of the origin~$O$. 
From~\cite{coo}, the exponential growth rate of~$G$ is positive and the limit-sup in its definition, \cf~\eqref{eq:om}, is a true limit.

The following classical definitions will be needed to state our main result.

\begin{definition} 
A metric space~$X$ is \emph{proper} if all its closed balls are compact and \emph{geodesic} if there is a geodesic segment joining every pair of points of~$X$.
A geodesic metric space~$X$ is \emph{$\delta$-hyperbolic} -- or \emph{Gromov hyperbolic} if the value of~$\delta$ does not matter -- if each side of a geodesic triangle of~$X$ is contained in the $4\delta$-neighborhood of the union of the other two sides, \cf~Section~\ref{sec:hyp} for an alternative definition.
A group~$G$ is \emph{elementary} if it contains a (finite or infinite) cyclic subgroup of finite index.
A group~$G$ acts \emph{properly} on a metric space~$X$ if for any compact set~$K \subset X$, there are only finitely many~$\alpha \in G$ such that $\alpha(K)$ intersects~$K$.
A group~$G$ acts \emph{cocompactly} on a metric space~$X$ if the quotient space~$X/G$ is compact.
\end{definition}


We can now state our main result.

\begin{theorem} \label{theo:main}
Let $G$ be a non-elementary group acting properly and cocompactly by isometries on a proper geodesic $\delta$-hyperbolic metric space~$X$.
Then $G$ is growth tight.
\end{theorem}

A quantitative version of this result can be found in the last section.

\medskip

Theorem~\ref{theo:main} yields an alternative proof of the growth tightness of non-elementary word hyperbolic groups in the sense of Gromov for word metrics, \cf~\cite{al}.
Indeed, every word hyperbolic group in the sense of Gromov acts properly and cocompactly by isometries on its Cayley graph (with respect to a given finite generating set) which forms a proper geodesic Gromov hyperbolic space.

Since the universal cover of a closed negatively curved $n$-manifold~$M$ is a Gromov hyperbolic space, we recover Theorem~\ref{theo:sam} as well by taking $G=\pi_1(M)$ and $X=\tilde{M}$.
Actually, Theorem~\ref{theo:main} allows us to extend this result in three directions as it applies to
\begin{itemize}
\item Riemannian metrics on~$M$ without curvature assumption (and even to Finsler metrics);
\item manifolds with a different topology type than~$M$, such as the non-aspherical manifolds $M \# (S^1 \times S^{n-1})$ or $M \# M_0$, where $M_0$ is any simply connected closed $n$-manifold different from~$S^n$;
\item intermediate covers and more generally covering towers, whereas Theorem~\ref{theo:sam} only deals with the universal cover of~$M$, see the following corollary for an illustration.
\end{itemize}

\begin{corollary} \label{coro:main}
Let $\hat{M}$ be a Riemannian normal cover of a closed Riemannian manifold~$M$.
Suppose that $\hat{M}$ is Gromov hyperbolic and that its boundary at infinity contains more than two points.
Then every Riemannian normal cover $\hat{M} \to \bar{M} \to M$ with $\hat{M} \neq \bar{M}$ satisfies $\omega(\bar{M}) < \omega(\hat{M})$.
\end{corollary}

Obviously, Corollary~\ref{coro:main} applies when $M$ is diffeomorphic to a closed locally symmetric manifold of negative curvature.
One may wonder if the conclusion holds true when $M$ is diffeomorphic to a closed irreductible higher rank locally symmetric manifold of noncompact type.
This question finds its answer in Margulis' normal subgroup theorem~\cite{mar}.
Indeed, in the higher rank case, the only normal covers of~$M$ are either compact (their exponential growth rate is zero) or are finitely covered by the universal cover~$\tilde{M}$ of~$M$ (their exponential growth rate agrees with the exponential growth rate of~$\tilde{M}$). 

In~\cite[\S5.1]{dpps}, building upon constructions of~\cite{dop}, the authors show that there exists a noncompact complete Riemannian manifold~$M$ with pinched negative curvature and finite volume which does not satisfy the conclusion of Corollary~\ref{coro:main} even for~$\hat{M}=\tilde{M}$.
This shows that we cannot replace the Gromov hyperbolic space~$X$ by a relatively hyperbolic metric space in Theorem~\ref{theo:main}. 

Finally, one may wonder if the gap between the exponential growth rates of~$G$ and~$\bar{G}$ is bounded away from zero.
Even in the case of word hyperbolic groups in the sense of Gromov, this is not true.
Indeed, it has recently been established in~\cite{cou} that the exponential growth rate of the periodic quotient~$G/G^n$ of a non-elementary torsion-free word hyperbolic group~$G$ is arbitrarily close to the exponential growth rate of~$G$, for every odd integer~$n$ large enough.
Here, $G^n$ represents the (normal) subgroup generated by the $n$-th powers of the elements of~$G$. \\

The strategy used to prove Theorem~\ref{theo:main} follows and extends the approach initiated by A.~Sambusetti and used in~\cite{sam01,sam02a,sam02b,sam03,sam04,sam08,dpps}.
However, the nature of the problem leads us to adopt a more global point of view, avoiding the use of any hyperbolic trigonometric comparison formula, which cannot be extended in the absence of curvature assumption and without a control of the topology of the spaces under consideration.
We present an outline of the proof in the next section.

\begin{notation} \label{notation}
Given $A,B,C \geq 0$, it is convenient to write $A \simeq B \pm C$ for $B-C \leq A \leq B+C$.
\end{notation}

\section{Outline of the proof}

In this section, we review the approach initiated by A.~Sambusetti and developed in this article. \\

Let $G$ be a non-elementary group acting properly and cocompactly by isometries on a proper geodesic $\delta$-hyperbolic metric space~$X$.
Every quotient group~$\bar{G}=G/N$ by a normal subgroup~$N \lhd G$ acts properly and cocompactly by isometries on the quotient metric space~$\bar{X} = X/N$.

Fix an origin~$O \in X$.
The left-invariant pseudo-distance~$d$, \cf~\eqref{eq:dg}, induces a semi-norm~$|| \cdot ||_G$ on~$G$ given by
\[
|| \alpha ||_G = d(e,\alpha)
\]
for every~$\alpha \in G$.
By definition, a semi-norm on~$G$ is a nonnegative function~\mbox{$|| \cdot ||$} defined on~$G$ such that 
\begin{align*}
||\alpha^{-1}|| & = || \alpha|| \\
||\alpha \beta|| & \leq ||\alpha || \, ||\beta||
\end{align*}
for every $\alpha, \beta \in G$.
Semi-norms and left-invariant pseudo-distances on a given group are in bijective correspondence.
Similarly, we define a semi-norm~\mbox{$|| \cdot ||_{\bar{G}}$} on~$\bar{G}$ from the quotient pseudo-distance~$\bar{d}$.
For the sake of simplicity and despite the risk of confusion, we will drop the prefixes pseudo- and semi- in the rest of this article and simply write ``distance" and ``norm".

For~$\lambda \geq 0$, consider the norm~$||\cdot||_\lambda$ on the free product~$\bar{G}*\Z_2$ where
\[
|| \gamma_1 * 1 * \cdots * \gamma_{m+1} ||_\lambda = \sum_{i=1}^{m+1} || \gamma_i||_{\bar{G}} + m\lambda
\]
for every element $\gamma_1 * 1 * \cdots * \gamma_{m+1} \in \bar{G}*\Z_2$, with~$\gamma_i \in \bar{G}$, in reduced form (that is, with $m$ minimal).
The norm~$|| \cdot ||_\lambda$ induces a left-invariant distance, denoted by~$d_\lambda$, on~$\bar{G}*\Z_2$.
We will write~$\omega(\bar{G}*\Z_2,\lambda)$ for the exponential growth rate of~$(\bar{G}*\Z_2,d_\lambda)$. 
Clearly, $\omega(\bar{G}* \Z_2,\lambda+\lambda') \leq \omega(\bar{G}* \Z_2,\lambda)$ for every $\lambda,\lambda' \geq 0$. \\

A direct combinatorial computation~\cite[Proposition~2.3]{sam02b} shows that 
\[
\omega(\bar{G}*\Z_2,\lambda) > \omega(\bar{G}).
\]
More precisely, the following estimate holds

\begin{proposition}[see Proposition~2.3 in~\cite{sam02b}] \label{prop:strict}
\mbox{ } 

For every~$\lambda >\diam(X/G)$, we have
\[
\omega(\bar{G}*\Z_2,\lambda) \geq \omega(\bar{G}) + \frac{1}{4 \lambda} \, \log ( 1+e^{- \lambda \omega(\bar{G})}).
\]
\end{proposition}

Strictly speaking this estimate has been stated for true distances, but it also applies to pseudo-distances. \\

If we could construct an injective $1$-Lipschitz map 
\[
(\bar{G}*\Z_2, || \cdot ||_{\lambda}) \to (G,|| \cdot ||_G),
\]
we could claim that the $R$-ball~$B_{\bar{G}*\Z_2,\lambda}(R)$ of~$\bar{G}*\Z_2$ for~$d_\lambda$ injects into an $R$-ball of~$G$.
This would imply that $\omega(\bar{G}*\Z_2,\lambda) \leq \omega(G)$.
Combined with Proposition~\ref{prop:strict}, the main theorem would follow. 
Here, we do not construct such a nonexpanding embedding, but derive a slightly weaker result which still leads to the desired result. \\


Fix $\rho > 0$.
Let $\bar{G}_{\rho}$ be a subset of~$\bar{G}$ containing the neutral element~$\bar{e} \in \bar{G}$ such that the elements of~$\bar{G}_\rho$ are at distance greater than~$\rho$ from each other and every element of~$\bar{G}$ is at distance at most~$\rho$ from an element of~$\bar{G}_{\rho}$.

Consider the subset~$\bar{G}_{\rho} * \Z_2$ of~$\bar{G}*\Z_2$ formed of the elements $\gamma_1 * 1 * \cdots * \gamma_{m+1}$ for $m \in \N$ with $\gamma_i \in \bar{G}_{\rho}$. \\

Suppose we can construct an injective $1$-Lipschitz map
\begin{equation} \label{eq:constphi}
\Phi:(\bar{G}_{\rho} * \Z_2, || \cdot ||_\lambda) \to (G, || \cdot ||_G)
\end{equation}
for $\lambda$ large enough.
Then, as previously, the $R$-ball~$B_{\bar{G}_{\rho} * \Z_2,\lambda}(R)$ injects into an $R$-ball of~$G$ and so 
\begin{equation} \label{eq:rho1}
\omega(\bar{G}_{\rho}*\Z_2,\lambda) \leq \omega(G).
\end{equation}
To derive the main theorem, we simply need to compare~$\omega(\bar{G}_{\rho}*\Z_2,\lambda)$ with~$\omega(\bar{G}*\Z_2,\lambda)$.
This is done in the next section, \cf~Proposition~\ref{prop:rho}, where we show that
\begin{equation} \label{eq:rho2}
\omega(\bar{G}*\Z_2,\lambda+\lambda') \leq \omega(\bar{G}_\rho*\Z_2,\tfrac{\lambda}{2})
\end{equation}
for $\lambda'$ large enough.
The combination of Proposition~\ref{prop:strict}, applied to~$\lambda+\lambda'$, with~\eqref{eq:rho2} and~\eqref{eq:rho1} allows us to conclude. \\

Thus, the key argument in the proof of the main theorem consists in constructing and deriving the properties of the nonexpanding map~$\Phi$, \cf~\eqref{eq:constphi}.
This is done in Proposition~\ref{prop:inj} for $\lambda$ large enough, without assuming that the action of~$G$ on~$X$ is cocompact.


\section{Exponential growth rate of lacunary subsets} \label{sec:lac}

The goal of this section is to compare the exponential growth rates of~\mbox{$\bar{G}*\Z_2$} and~$\bar{G}_{\rho}*\Z_2$.
We will use the notations (and obvious extensions) previously introduced without further notice. \\


Given $\rho > 0$, there exists~$r_{\rho} >0$ such that
\begin{equation} \label{eq:boundR}
\card B_{\bar{G}_{\rho}}(r_{\rho}) \geq \card B_{\bar{G}}(\rho).
\end{equation}

An explicit value for~$r_{\rho}$ can be obtained from a (naive) packing argument based on the following observation: every point of~$\bar{X}$ is at distance at most~$\Delta$ from a point of the $\bar{G}$-orbit of~$\bar{O}$ and so at distance at most~$\Delta+\rho$ from some point~$\mathring{\gamma}(\bar{O})$, where $\mathring{\gamma} \in \bar{G}_{\rho}$.
Here, $\Delta = \diam(\bar{X}/\bar{G}) = \diam(X/G)$ and $\bar{O}$ represents the projection of~$O$ to~$\bar{X}$.
Let $\epsilon >0$.
Take $\card B_{\bar{G}}(\rho)$ points~$(\bar{x}_i)$ on a minimizing ray of~$\bar{X}$ of length~$\ell=2 (\Delta+ \rho) \, \card B_{\bar{G}}(\rho)+\epsilon$ based at~$\bar{O}$ with $d(\bar{x}_i,\bar{x}_j) > 2(\Delta+\rho)$ for $i \neq j$.
From the previous observation, every point~$\bar{x}_i$ is at distance at most~$\Delta+\rho$ from some~$\mathring{\gamma}_i(\bar{O})$, where $\mathring{\gamma_i} \in \bar{G}_{\rho}$.
By construction, the elements~$\mathring{\gamma}_i$ are disjoint and lie in~$B_{\bar{G}_{\rho}}(\ell+\Delta+\rho)$.
Thus, the bound~\eqref{eq:boundR} holds with
\[
r_{\rho} = 3 (\Delta+ \rho) \, \card B_{\bar{G}}(\rho).
\]
By applying the previous observation to a point of~$\bar{X}$ at distance~$\Delta + \rho + \epsilon$ from~$\bar{O}$, we can also show that there exists~$\mathring{\theta} \in \bar{G}_{\rho}$ different from~$\bar{e}$ with $||\mathring{\theta}||_{\bar{G}} \leq 2(\Delta+\rho)$.

\begin{proposition} \label{prop:rho}
We have
\[
\omega({\bar{G}*\Z_2,\lambda+ \lambda' }) \leq \omega(\bar{G}_{\rho}*\Z_2, \tfrac{\lambda}{2})
\]
where $\lambda \geq 0$ and $\lambda' \geq R_\rho = 15(\Delta+\rho) \, \card B_{\bar{G}}(3(\Delta+\rho))$.
\end{proposition}

\begin{proof}
Let $\lambda' \geq r_{\sigma} + \sigma$ where $\sigma=3(\Delta + \rho)$.
Consider the map $\varphi : \bar{G}*\Z_2 \to \bar{G}_{\rho} * \Z_2$ which takes every element~$\gamma_1 * 1*\cdots * \gamma_{m+1} \in\bar{G}*\Z_2$ in reduced form to an element~$\mathring{\gamma}_1 * 1* \cdots * \mathring{\gamma}_{m+1} \in \bar{G}_{\rho} * \Z_2$, where~$\mathring{\gamma}_i$ is equal to~$\mathring{\theta}$ if $\gamma_i \in B_{\bar{G}}(\rho)$ and $\mathring{\gamma}_i$ is an element of~$\bar{G}_{\rho}$ at distance at most~$\rho$ from~$\gamma_i$ otherwise.
As no~$\mathring{\gamma}_i$ agrees with~$\bar{e}$ (this is the reason for introducing~$\mathring{\theta}$), the product $\mathring{\gamma}_1 * 1* \cdots * \mathring{\gamma}_{m+1}$ is in reduced form.
Note also that the distance between~$\gamma_i$ and~$\mathring{\gamma}_i$ is at most~$2(\Delta + \rho) + \rho \leq 3( \Delta + \rho) = \sigma$.

The map~$\varphi$ sends~$B_{\bar{G}*\Z_2,\lambda+\lambda'}(R)$ to~$B_{\bar{G}_{\rho}*\Z_2,\lambda+\lambda'-\sigma}(R+\sigma)$.
Furthermore, an element~$\mathring{\gamma}_1 * 1* \cdots * \mathring{\gamma}_{m+1} \in \bar{G}_{\rho}*\Z_2$ in reduced form has at most
\[
\card B_{\bar{G}}(\mathring{\gamma}_1,\sigma) \times \cdots \times \card B_{\bar{G}}(\mathring{\gamma}_{m+1},\sigma) = \left( \card B_{\bar{G}}(\sigma)\right)^{m+1}
\]
antecedents by~$\varphi$.
Hence, the cardinal of~$B_{\bar{G}*\Z_2,\lambda+\lambda'}(R)$ is bounded by the following sum
\[
\sum_{m=0}^\infty \card \{ \mathring{\gamma}_1 * 1* \cdots * \mathring{\gamma}_{m+1} \in B_{\bar{G}_{\rho}*\Z_2,\lambda+\lambda'-\sigma}(R+\sigma) \} \,  \left( \card B_{\bar{G}}(\sigma)\right)^{m+1}.
\]
Now, since $\card B_{\bar{G}}(\sigma) \leq \card B_{\bar{G}_{\rho}}(r_{\sigma})$, each term of this sum is bounded by the number of elements
\[
\mathring{\gamma}_1 * 1* \cdots * \mathring{\gamma}_{m+1}*1*\mathring{c}_1 * 1* \cdots * \mathring{c}_{m+1} \in \bar{G}_{\rho} * \Z_2
\]
such that
\[
\sum_{i=1}^{m+1} || \mathring{\gamma}_i ||_{\bar{G}} + m (\lambda + \lambda' -\sigma) \leq R+\sigma \mbox{ and } \mathring{c}_i \in B_{\bar{G}_{\rho}}(r_{\sigma}).
\]
Since all these elements satisfy
\[
\sum_{i=1}^{m+1} || \mathring{\gamma}_i ||_{\bar{G}} + \sum_{i=1}^{m+1} || \mathring{c}_i ||_{\bar{G}} + m (\lambda + \lambda' -\sigma - r_{\sigma}) \leq R+\sigma+r_{\sigma} 
\]
and so
\[
\sum_{i=1}^{m+1} || \mathring{\gamma}_i ||_{\bar{G}} + \sum_{i=1}^{m+1} || \mathring{c}_i ||_{\bar{G}} + (2m+1) \, \frac{1}{2} (\lambda + \lambda' -\sigma - r_{\sigma}) \leq R+\lambda + \lambda',
\]
we derive that
\[
\card B_{\bar{G}*\Z_2,\lambda+\lambda'}(R) \leq \card B_{\bar{G}_{\rho}*\Z_2, \frac{1}{2} (\lambda + \lambda' -\sigma - r_{\sigma})}(R+\lambda+\lambda').
\]
Therefore, $\omega({\bar{G}*\Z_2,\lambda+ \lambda'}) \leq \omega(\bar{G}_{\rho}*\Z_2, \frac{\lambda}{2})$.
\end{proof}

\section{Classical results about Gromov hyperbolic spaces} \label{sec:hyp}

In this section, we recall the definition of Gromov hyperbolic spaces and present some well-known results.
Classical references on the subject include~\cite{gro,gh,cdp}. 

\begin{definition} 
Let $X=(X,|\cdot|)$ be a metric space.
The \emph{Gromov product} of $x,y \in X$ with respect to a basepoint~$w \in X$ is defined as
\begin{equation} \label{eq:prod}
(x | y )_w = \frac{1}{2} (|xw|+|yw|-|xy|).
\end{equation}
By the triangle inequality, it is nonnegative.

Fix $\delta \geq 0$. 
A metric space~$X$ is \emph{$\delta$-hyperbolic} if
\[
(x | y)_w \geq \min \{ (x | z)_w, (y | z)_w \} - \delta
\]
for every $x,y,z,w \in X$.
Equivalently, a metric space~$X$ is \emph{$\delta$-hyperbolic} if
\begin{equation} \label{eq:4}
|xy| + |zw| \geq \max \{ |xz|+|yw|,|yz|+|xw| \} + 2\delta
\end{equation}
for every $x,y,z,w \in X$.

A metric space is \emph{Gromov hyperbolic} if it is $\delta$-hyperbolic for some $\delta \geq 0$.
A finitely generated group is \emph{word hyperbolic} in the sense of Gromov if its Cayley graph with respect to some (or equivalently any) finite generating set is Gromov hyperbolic.
\end{definition}

\begin{example}
Gromov hyperbolic spaces include complete simply connected Riemannian manifolds of sectional curvature bounded away from zero and their convex subsets, metric trees and more generally ${\rm CAT}(-1)$ spaces.
\end{example}

\begin{example} \label{ex:word}
Let $G$ be a group acting properly and cocompactly by isometries on a proper geodesic $\delta$-hyperbolic metric space~$X$ (as in Theorem~\ref{theo:main}).
Then the group~$G$ is finitely generated.
Furthermore, its Cayley graph with respect to any finite generating set is quasi-isometric to~$X$ and so is Gromov hyperbolic.
Thus, $G$ is a word hyperbolic group in the sense of Gromov.
\end{example}

\begin{remark}
In this definition, the metric space~$X$ is not required to be geodesic.
However, from~\cite{bs}, every $\delta$-hyperbolic metric space isometrically embeds into a complete geodesic $\delta$-hyperbolic metric space.
\end{remark}

Without loss of generality, we will assume in the sequel that $X$ is a complete geodesic $\delta$-hyperbolic metric spaces. \\

The following result about approximation maps can be found in~\cite{gro}, \cite[Ch.~2, \S3]{gh} and~\cite{cdp}.

\begin{lemma} \label{lem:approx}
Given a geodesic triangle~$\Delta$ in~$X$, there exists a map \mbox{$\Phi: \Delta \to T$} to a possibly degenerated tripode~$T$ (\ie, a metric tree with at most three leaves) such that
\begin{enumerate}
\item the restriction of~$\Phi$ to each edge of~$\Delta$ is an isometry, \label{approx1}
\item for every $x,y \in X$,
\[
|xy|-4 \delta \leq |\Phi(x) \Phi(y)| \leq |xy|.
\]
\label{approx2}
\end{enumerate}
\end{lemma}

In this lemma, we also denoted by~$|\cdot|$ the metric on~$T$.
We will refer to the map~$\Phi$ as the \emph{approximation map} of the geodesic triangle~$\Delta$.

\begin{remark}
This result implies the Rips condition: each side of a geodesic triangle of~$X$ is contained in the $4\delta$-neighborhood of the union of the other two sides.
\end{remark}

\begin{remark} \label{rem:prod}
Given $x,y,z \in X$, let $\Phi$ be the approximation map of a geodesic triangle~$\Delta=\Delta(x,y,z)$ with vertices $x$, $y$ and~$z$.
From Lemma~\ref{lem:approx}.\eqref{approx1}, the Gromov product~$(x | y )_z$ is equal to the distance between~$\Phi(z)$ and the center of the tripode~$T$.
\end{remark}

The following lemma is a simple version of the Local-to-Global theorem, \cf~\cite{gro,gh,cdp}.

\begin{lemma} \label{lem:14}
Let $x,y,p,q \in X$ such that 
\[
(x | q)_p \leq 4 \delta, \quad
(y | p)_q \leq 4 \delta, \quad
|xp| \leq |xq| \ \mbox{ and } \ 
|yq| \leq |yp|.
\]
Then
\[
|xy| \geq |xp| + |pq| + |qy| - 14 \delta.
\]
\end{lemma}

\begin{proof}
From the definition of the Gromov product, \cf~\eqref{eq:prod}, the bounds $(x | q)_p \leq 4 \delta$ and $(y | p)_q \leq 4 \delta$ yield the following two estimates
\begin{align*}
|xq| & \geq |xp| + |pq|- 8 \delta  \\
|yp| & \geq |yq| + |qp|- 8 \delta.
\end{align*}
Now, from~\eqref{eq:4}, we have
\begin{align*}
|xy| + |pq| & \geq \max \{ |xp|+|yq|, |yp| + |xq| \} + 2 \delta \\
 & \geq |xq| + |yp| + 2 \delta.
\end{align*}
Combined with the previous two estimates, we obtain the desired lower bound for~$|xy|$.
\end{proof}

The following simple fact will be useful in the sequel.

\begin{lemma} \label{lem:4d}
Let $x \in X$.
Consider the projection~$p$ of~$x$ to a given geodesic line or geodesic segment~$\tau$ (the projection may not be unique).
Then, for every $q \in \tau$,
\[
(x | q)_p \leq 4 \delta.
\]
In particular,
\[
|xq| \geq |xp| + |pq|-8 \delta.
\]
\end{lemma}

\begin{proof}
Consider an approximation map $\Phi:\Delta \to T$ of a geodesic triangle~$\Delta=\Delta(x,p,q)$ with vertices $x$, $p$ and~$q$.
The center of the tripode~$T$ has three preimages by~$\Phi$, one in each segment $[x,p]$, $[p,q]$ and~$[x,q]$.
Let $x_1$ and~$x_2$ be the preimages of the center of~$T$ in~$[xp]$ and~$[pq]$.
Note that $|xx_1| = |xp|-|x_1p|$.
From Lemma~\ref{lem:approx}.\eqref{approx2}, the points~$x_1$ and~$x_2$ are at distance at most~$4 \delta$ , that is, $|x_1x_2| \leq 4 \delta$.
Since $p$ is the projection of~$x$ to~$\tau$ and $x_2$ lies in~$\tau$, we have $|xp| \leq |xx_2|$.
Combining these estimates with the triangular inequality, we obtain
\begin{align*}
|xp| \leq |xx_2| &\leq |xx_1| + |x_1 x_2| \\
 &\leq |xp|-|x_1 p| +4 \delta.
\end{align*}
That is, $|x_1 p| \leq 4 \delta$.
Now, we observed in Remark~\ref{rem:prod} that $|x_1 p| = (x | q)_p$.
The result follows.
\end{proof}

The following classical result follows from the previous lemma. 

\begin{lemma} \label{lem:ell}
Let $[x,y]$ be a geodesic segment joining~$x$ to~$y$ in~$X$.
Consider an arc~$\gamma$ in~$X$ with the same endpoints as~$[x,y]$ such that 
\begin{equation} \label{eq:ell}
\length(\gamma) \leq |xy| + \ell.
\end{equation}
Then the arc~$\gamma$ lies in the $ \left( \frac{\ell}{2} + 8\delta \right)$-neighborhood of~$[x,y]$.
\end{lemma}

\begin{proof}
Let $p$ be the projection to~$[x,y]$ of a point~$z$ in~$\gamma$.
From Lemma~\ref{lem:4d}, we have the following two estimates
\begin{align*}
|xz| & \geq |xp| + |pz| - 8 \delta \\
|zy| & \geq |zp| + |py| - 8 \delta.
\end{align*}
Since $p$ lies between~$x$ and~$y$, we have $|xy| = |xp| + |py|$.
Hence,
\begin{align*}
\length(\gamma) & \geq |xz| + |zy| \\
 & \geq |xp| + |pz| - 8 \delta + |zp| + |py| - 8 \delta \\
 & \geq 2 \, |zp| + |xy| - 16 \delta.
\end{align*}
Substituting this inequality into the upper bound~\eqref{eq:ell}, we obtain
\[
d(z,[x,y]) = |zp| \geq \frac{\ell}{2} + 8 \delta.
\]
\end{proof}

\section{Traveling along hyperbolic orbits}

In this section, we introduce some definitions, notations and constructions which will be used throughout this article. \\

The following set of definitions and properties can be found in~\cite{gro,gh,cdp}.

\begin{definition}
Let $X$ be a proper geodesic $\delta$-hyperbolic metric space.
The \emph{boundary at infinity} of~$X$, denoted by~$\partial X$, is defined as the equivalence classes of geodesic rays of~$X$, where two rays are equivalent if they are at finite Hausdorff distance.
The space~$X \cup \partial X$, endowed with the natural topology extending the initial topology on~$X$, is compact and contains~$X$ as an open dense subset.
Given two distinct points~$a$ and~$b$ in~$\partial X$, there is a geodesic line~$\tau$ in~$X$ joining~$a$ to~$b$, that is, $\tau(-\infty)=a$ and~$\tau(\infty)=b$.
Every isometry of~$X$ uniquely extends to a homeomorphism of~$X \cup \partial X$.
An isometry~$\alpha$ of~$X$ is \emph{hyperbolic} if for some (or equivalently any) point~$x \in X$, the map $n \mapsto \alpha^n(x)$ is a quasi-isometric embedding of~$\Z$ into~$X$.
Alternatively, an isometry of~$X$ is hyperbolic if and only if it is of infinite order.
Every hyperbolic isometry of~$X$ has exactly two fixed points on~$\partial X$.
An \emph{axis} of a hyperbolic isometry of~$X$ is a geodesic line of~$X$ joining the two fixed points of the isometry.
The \emph{minimal displacement} of an isometry~$\alpha$ of~$X$ is defined as
\[
\dis(\alpha) = \inf_{x \in X} |x \alpha(x)|.
\]
The minimal displacement of a hyperbolic isometry of~$X$ is positive.
\end{definition}

\begin{example} \label{ex:hyp}
Let $G$ be a group acting properly and cocompactly by isometries on a proper geodesic $\delta$-hyperbolic metric space~$X$.
We know that $G$ is a word hyperbolic group in the sense of Gromov, \cf~Example~\ref{ex:word}.
The argument leading to this result also shows that an element of~$G$ is hyperbolic as an isometry of~$X$ if and only if it is hyperbolic as an isometry of the Cayley graph of~$G$ with respect to any finite generating set (and so if and only if it is of infinite order).
The boundary at infinity of the word hyperbolic group~$G$, defined as the boundary at infinity of its Cayley graph with respect to some (and so any) finite generating set, agrees with~$\partial X$.
The group~$G$ is non-elementary if and only if $\partial X$ contains more than two points.
In this case, $G$ contains a hyperbolic element.
More generally, every infinite subgroup~$N$ of a word hyperbolic group~$G$ in the sense of Gromov contains a hyperbolic element.
\end{example}

Let $G$ be a finitely generated group acting by isometries on a proper geodesic $\delta$-hyperbolic metric space~$X$.
Suppose that $G$ contains a hyperbolic isometry $\xi$ of~$X$, \cf~Example~\ref{ex:hyp}.
By taking a large enough power of~$\xi$ if necessary, we can assume that the minimal displacement of~$\xi$, denoted by~$L=\dis(\xi)$, is at most~$300 \delta$, \ie, $L \geq 300 \delta$.
Fix an axis~$\tau$ of~$\xi$.

Given $\epsilon \in (0,\delta)$, let $O \in X$ be an origin with $|O \xi(O)| \leq L + \epsilon$.

\begin{lemma} \label{lem:28}
The point~$O$ is at distance at most~$28 \delta$ from~$\tau$.
\end{lemma}

\begin{proof}
Let $p$ and~$q$ be the projections of~$O$ and~$\xi(O)$ to~$\tau$ (they may not be unique but we choose some).
Let also $q_-$ be the projection of~$\xi^{-1}(q)$ to~$\tau$.
From~\cite[Corollaire~7.3]{gh}, the image~$\xi^{-1}(\tau)$, which is a geodesic line with the same endpoints at infinity as~$\tau$, is at distance at most~$16 \delta$ from~$\tau$.
Thus, 
\begin{equation} \label{eq:qq-}
| \xi^{-1}(q) q_-| \leq 16 \delta.
\end{equation}
From this relation and Lemma~\ref{lem:4d}, we derive
\[
|Op|+|pq_-| -8\delta \leq |Oq_-| \leq |O\xi^{-1}(q)| + 16 \delta.
\]
That is,
\begin{equation} \label{eq:q-}
|pq_-| \leq |\xi(O)q| - |Op|+24 \delta.
\end{equation}
Now, from Lemma~\ref{lem:14}, we have
\begin{align*}
|O\xi(O)| & \geq |Op|+|pq|+|q\xi(O)| - 14 \delta \\
& \geq |Op|+|\xi^{-1}(q)q|-|\xi^{-1}(q)p|+|q\xi(O)| - 14 \delta \\
& \geq |Op|+|q\xi(q)|-|pq_-|-16\delta+|q\xi(O)| - 14 \delta \\
& \geq |q\xi(q)|+2|Op|-54\delta
\end{align*}
where the third inequality follows from~\eqref{eq:qq-} and the last one from~\eqref{eq:q-}.
By definition of~$O$, we obtain
\[
|Op| \leq 27 \delta + \frac{\epsilon}{2} < 28 \delta.
\]
\end{proof}

Let~$x_1 =O$.
For $i \in \Z^* = \Z \setminus \{ 0 \}$, we denote by~$x_i$ the point~$\xi^{i-1}(x_1)$ if~$i >0$ and $\xi^i(x_1)$ if~$i<0$ (this choice of indices may not seem natural, but it allows us to consider fewer cases in forecoming arguments).
We also define~$p_i$ as the projection of~$x_i$ to~$\tau$ (again, it may not be unique but we choose one).
Since $\xi$ is an isometry, the points~$x_i$ of the $\xi$-orbit of~$x_1$ attain the minimal displacement of~$\xi$ up to~$\epsilon$.
Thus, from Lemma~\ref{lem:28}, we derive
\begin{equation} \label{eq:16}
|x_i p_i| \leq 28 \delta.
\end{equation}
For every pair of indices $i,j \in \Z^*$,  we obtain
\[
|p_i p_j| \geq | x_i x_j| - 56 \delta \geq L - ( 56 \delta + \epsilon).
\]
If the indices $i$ and~$j$ are adjacent, we actually have
\begin{align}
|p_i p_j| & \simeq |x_i x_j| \pm 56 \delta \nonumber \\
& \simeq L \pm (56 \delta + \epsilon).  \label{eq:pij}
\end{align}
using the notation~\ref{notation}.
Therefore, since $L > 168 \delta + 3 \epsilon$, we deduce that the points~$p_i$ lie in~$\tau$ in the order induced by~$\Z^*$. \\

Consider the Voronoi cells of the $\xi$-orbit~$(x_i)$ of~$O$, namely
\[
D_i = \{ x \in X \mid |xx_i| \leq |xx_j| \mbox{ for every } j \in \Z^* \}.
\]
Note that $\xi(D_i) = D_{i+1}$ if $i \in \Z^* \setminus \{-1\}$ and $\xi(D_{-1}) = D_1$. \\


\forget
\begin{lemma} \label{lem:L}
Let $x \in D_i$ and $p$ be a projection of~$x$ to~$\tau$.
Then
\begin{equation} \label{eq:ppi}
|p_i p| \leq \frac{L}{2} + 36 \delta + \epsilon.
\end{equation}
In particular, for every $p_j \neq p_i$, we have
\[
|p_jp| \geq \frac{L}{2} - 68 \delta - 2 \epsilon.
\]
\end{lemma}

\begin{proof}
The idea of the proof goes as follows.
If $p$ is far enough from~$p_i$ (and so from~$x_i$), then $p$ is much closer to some adjacent point~$p_j$ of~$p_i$ than to~$p_i$.
This implies that $x$ is closer to~$x_j$ than to~$x_i$, which is impossible since $x \in D_i$.

Let us give some precise estimates.
Arguing by contradiction, we assume that 
\[
|p_i p| > \frac{L}{2} + 36 \delta + \epsilon,
\]
that is,
\begin{equation} \label{eq:L}
L < 2 \, |p p_i| - 72 \delta - 2 \epsilon.
\end{equation}

Suppose there is an index~$j$ in~$\Z^*$ adjacent to~$i$ such that $p$ lies between $p_i$ and~$p_j$.
Then $|pp_j| = |p_ip_j| - |pp_i|$.
From the estimates~\eqref{eq:pij} and~\eqref{eq:L}, we deduce that 
\[
|p_ip_j| \leq 2 \, |p p_i| - 40 \delta - \epsilon.
\]
Hence, 
\[
|pp_j| < |pp_i| - 40 \delta.
\]

Suppose, in the opposite case, that there is an index~$j$ in~$\Z^*$ adjacent to~$i$ such that $p_j$ lies between $p_i$ and~$p$.
Then, still from~\eqref{eq:pij}, we have
\begin{align*}
|pp_j| & = |pp_i| - |p_ip_j| \\
& \leq |p p_i| - L + 32 \delta + \epsilon.
\end{align*}

In any case (recall that $L \geq 100 \delta$), there exists an index~$j$ in~$\Z^*$ adjacent to~$i$ such that
\[
|pp_j| < |pp_i| - 40 \delta.
\]
From the triangular inequality, the previous bound and~\eqref{eq:16}, we have
\begin{align*}
|xx_j| & \leq |xp| + |pp_j| + |p_j x_j| \\
 & < |xp| + |pp_i| - 40 \delta + 16 \delta.
\end{align*}
On the other hand, from Lemma~\ref{lem:4d}, we have
\[
|xp|+|pp_i| \leq |xp_i| + 8 \delta \leq |x x_i| + |x_ip_i|+8 \delta.
\]
Hence, with the help of~\eqref{eq:16}, we obtain
\[
|xx_j| < |xx_i| + 16 \delta + 8 \delta -40 \delta + 16 \delta  = |xx_i|.
\]
This is impossible since $x \in D_i$.
Therefore, the inequality~\eqref{eq:ppi} holds true. \\

The second part of the lemma follows from the triangular inequality 
\[
|p_jp| \geq |p_jp_i| - |p_ip| 
\]
combined with the estimates~\eqref{eq:pij} and~\eqref{eq:ppi}.
\end{proof}
\forgotten

\begin{lemma} \label{lem:L}
Given $x \in D_i$, let $p$ be a projection of~$x$ to~$\tau$.
Then $p$ lies between the points $p_{i_-}$ and~$p_{i_+}$ of~$\{ p_j \mid j \in \Z^*, j \neq i \}$  adjacent to~$p_i$.

In particular,
\[
|p_i p| \leq L + 56 \delta + \epsilon.
\]
\end{lemma}

\begin{proof}
By contradiction, we assume that there is an index~$j \in \Z^*$ adjacent to~$i$ such that $p_j$ lies between $p_i$ and~$p$.
From~\eqref{eq:pij}, we have
\begin{align*}
|pp_j| & = |pp_i| - |p_ip_j| \\
& \leq |p p_i| - L + 56 \delta + \epsilon.
\end{align*}
From the triangular inequality and the bound~\eqref{eq:16}, this estimate leads to
\begin{align*}
|xx_j| & \leq |xp| + |pp_j| + |p_j x_j| \\
 & \leq |xp| + |pp_i| - L + 56 \delta + \epsilon + 28 \delta.
\end{align*}
On the other hand, from Lemma~\ref{lem:4d}, we have
\[
|xp|+|pp_i| \leq |xp_i| + 8 \delta \leq |x x_i| + |x_ip_i|+8 \delta.
\]
Hence, with the help of~\eqref{eq:16} and the inequality $L > 120 \delta + \epsilon$, we obtain
\[
|xx_j| \leq |xx_i| + 28 \delta + 8 \delta - L + 56 \delta + \epsilon + 28 \delta  < |xx_i|.
\]
Thus, $x$ does not lie in~$D_i$, which is absurd.
Hence the first part of the lemma.

The second part of the lemma follows from~\eqref{eq:pij}.
\end{proof}

\begin{remark}
Lemma~\ref{lem:L} also shows that two domains~$D_i$ and~$D_j$ corresponding to non-adjacent indices are disjoint.
Thus, the Voronoi cells are ordered by their indices.
\end{remark} 

\begin{lemma} \label{lem:+}
Let $x$ and $y$ be two points of~$X$ separated by~$D_{\pm 1}$ or~$D_{\pm 2}$.
Then
\[
|xy| \geq |O x| + |O y| - 4L - 294 \delta -4 \epsilon. 
\]
\end{lemma}

\begin{proof}
Let $p$ and~$q$ be the projections of~$x$ and~$y$ to~$\tau$.
By assumption, there exist two indices~$i, j \in \Z^*$ such that $x \in D_i$ and~$y \in D_j$ with $D_{\pm 1}$ or~$D_{\pm 2}$ separating~$D_i$ and~$D_j$.
Observe that the indices~$i$ and $j$ cannot be both less than~$-2$ or greater than~$2$.
Thus, two cases may occur: either $\pm 1$ or~$\pm 2$ separates the two indices, or one of them is equal to~$\pm 1$ or~$\pm 2$.

In the former case, we derive from Lemma~\ref{lem:L} that $p_{\pm 1}$ or~$p_{\pm 2}$ separates~$p$ and~$q$.
This point between~$p$ and~$q$ will be denoted by~$r$.
From the bounds~\eqref{eq:pij} and~\eqref{eq:16}, we derive $|Or| \leq 2L+140 \delta + 2 \epsilon$.
Hence
\begin{align*}
|pq| & = |pr|+|rq| \\
 & \geq |Op| + |Oq| - 2 \, |O r| \\
 & \geq |Op| + |Oq| - 2 (2L+140 \delta + 2\epsilon).
\end{align*}

In the latter case, we deduce as previously that $|Op|$ or~$|Oq|$ is bounded from above by $2(2L + 140 \delta + 2\epsilon)$.
Thus,
\[
|pq| \geq |Op| + |Oq| - 2 (2L+140 \delta + 2\epsilon).
\]

Now, from Lemmas~\ref{lem:4d} and~\ref{lem:14}, we derive in both cases that
\begin{align*}
|xy| & \geq |xp| + |pq| + |qy| - 14 \delta \\
& \geq |xp| + |Op| + |Oq| + |qy| - 2 (2L+140 \delta + 2\epsilon) - 14 \delta \\
& \geq |Ox| + |Oy| - 4L-294 \delta -4 \epsilon.
\end{align*}
\end{proof}

\section{Geometric properties of symmetric elements}

Using the previous notations and constructions, we introduce a notion of symmetric element in~$G$ and establish some (almost) norm-preserving properties.

\begin{definition}  \label{def:sign}
Let $\beta \in G$.
Denote by $j(\beta)$ the smallest index $j \in \Z^*$ such that $\beta(O) \in D_j$.
We say that $\beta$ is \emph{positive} if $j(\beta)>0$ and \emph{negative} if $j(\beta)<0$.

We define $\beta_{\pm} \in G$ as follows
\[
\beta_+ = \beta
\]
and
\[
\beta_- = \left\{
\begin{array}{l}
\xi^{-2j-1} \beta \quad \mbox{ if } \beta \mbox{ is positive} \\

\xi^{-2j+1} \beta \quad \mbox{ if } \beta \mbox{ is negative}
\end{array}
\right.
\]
where $j = j(\beta)$.

We will think of $\beta_-$ as the symmetric element of~$\beta$ with respect to a symmetry parallel to~$\xi$.

Consider the following norm on~$G$ defined for every~$\beta \in G$ as
\[
||\beta|| = |O \beta(O)|.
\]
\end{definition}

\begin{proposition} \label{prop:-}
Let $\beta,\beta_1,\beta_2 \in G$ and $j = j(\beta)$.
\begin{enumerate}
\item $\beta_-(O)$ lies in $D_{-j-2}$ if $j>0$ and in $D_{-j+2}$ if $j<0$. \label{-1} \\
In particular, $\beta_+$ and $\beta_-$ have opposite signs. \label{opp}
\item If $(\beta_1)_- = (\beta_2)_-$ then $\beta_1 = \beta_2$. \label{-=}
\item We have
\[
||\beta_-|| \simeq ||\beta|| \pm \Delta_-
\]
where $\Delta_-=\Delta_-(L,\delta,\epsilon) = 8L + 464 \delta + 8 \epsilon \leq 10 L$. \label{beta-}
\end{enumerate}
\end{proposition}

\begin{proof}
The point~\eqref{-1} is obvious by construction of~$\beta_-$.

For the second point, if $(\beta_1)_- = (\beta_2)_-$ then $\beta_1$ and~$\beta_2$ have the same sign from the second assertion of~\eqref{-1}.
We derive from the definition of~$\beta_-$ that $\xi^k \beta_1 = \xi^k \beta_2$ for $k=-2j \pm 1$.
Hence $\beta_1=\beta_2$.

For the last point, let $x=\beta(O)$ and~$x_-=\beta_-(O)_- = \xi^{-2j \mp 1} \beta(O)$, where the exponent $-2j \mp 1$ depends on the sign of~$\beta$, see~Definition~\ref{def:sign}.
Fix $k = -j \mp 2$.
By assumption, $x \in D_j$ and from the point~\eqref{-1}, $x_- \in D_k$.
Thus, $|xx_j| = |x_- x_k|$.

Let $p$ and~$p_-$ be the projections of~$x$ and~$x_-$ to~$\tau$.
From Lemma~\ref{lem:L} and the bound~\eqref{eq:16}, both $|x_j p|$ and~$|x_k p_-|$ are bounded from above by $L+84 \delta + \epsilon$.
Combined with
\begin{align*}
|x p| & \simeq |xx_j| \pm |x_j p| \\
|x_- p_-| & \simeq |x_- x_k| \pm |x_k p_-|
\end{align*}
and the relation $|xx_j| = |x_- x_k|$, we obtain
\begin{equation} \label{eq:xx-}
|xp| \simeq |x_- p_-| \pm (2L+ 168 \delta + 2 \epsilon).
\end{equation}

On the other hand, from~\eqref{eq:16} and the previous bound on~$|x_j p|$, we derive
\begin{align}
|p_1 p| & \simeq |x_1 x_j| \pm ( |p_1 x_1| + | x_j p|) \nonumber \\
& \simeq |x_1 x_j| \pm ( L + 112 \delta + \epsilon) \label{eq:p0p}
\end{align}
Similarly,
\begin{equation} \label{eq:p0p-}
|p_1 p_-| \simeq |x_1 x_k| \pm ( L + 112 \delta + \epsilon).
\end{equation}
Now, since 
\begin{align*}
|x_1 x_j| & = |x_{-j+2} x_1| \\
& \simeq |x_k x_1| \pm 4(L+ \epsilon)
\end{align*}
we obtain from~\eqref{eq:p0p} and~\eqref{eq:p0p-} that
\begin{equation} \label{eq:pp-}
|p_1 p| \simeq |p_1 p_-| \pm (6L+224 \delta + 6 \epsilon).
\end{equation}

Now, from Lemma~\ref{lem:4d}, we have
\begin{align*}
|xp_1| & \simeq |xp| + |pp_1| \pm 8 \delta \\
|x_-p_1| & \simeq |x_-p_-| + |p_-p_1| \pm 8 \delta
\end{align*}
Combined with~\eqref{eq:xx-} and~\eqref{eq:pp-}, we obtain
\[
|xp_1| \simeq |x_-p_1| \pm ( 8L + 408 \delta + 8 \epsilon).
\]
The result follows from the bound~$|Op_1| \leq 28 \delta$ derived in~\eqref{eq:16}.
\end{proof}

\forget
 and $p_- \in D_{-j \mp 2}$ depending on the sign of~$j$.
The isometries~$\xi^k_0$ and~$\xi^{-k}_0$ respectively take~$x$ to~$x_-$ and~$x_-$ to~$x$.
They also keep the endpoints at infinity of~$\tau$ fixed.
In particular, both isometries send $\tau$ to some geodesic lines with the same endpoints at infinity as~$\tau$.
From~\cite[Chap.~7, Corollaire~3]{gh}, these geodesic lines are at distance at most~$16 \delta$ from~$\tau$.
As $p$ and~$p_-$ lie in~$\tau$, the images $\xi^k_0(p)$ and~$\xi^{-k}_0(p_-)$ are at distance at most~$16 \delta$ from~$\tau$.
We will denote by~$q$ and~$q_-$ the projections of~$\xi^{-k}_0(p_-)$ and~$\xi^k_0(p)$ to~$\tau$.

On the one hand, we have
\begin{eqnarray*}
|xp| = |x_- \xi^k_0(p)| & \geq & |x_- q_-| - 16 \delta \\
 & \geq & |x_- p_-| - 16 \delta
\end{eqnarray*}
Similarly, we obtain
\[
|x_- p_-| \geq |xp| - 16 \delta
\]
Hence,
\begin{equation} \label{eq:px}
|x_- p_-| \simeq |xp| \pm \leq 16 \delta.
\end{equation}

On the other hand, by the first and second triangular inequalities, we have
\[
|x_0 \xi^j_0(x_0)| - |\xi^j_0(x_0) x_j| - |x_jp| \leq |x_0p| \leq  |x_0 \xi^j_0(x_0)| + |\xi^j_0(x_0) x_j| + |x_jp|.
\]
Recall that $|\xi^j_0(x_0) x_j| \leq 16 \delta$.
Since $p \in \tau_j$, $|x_j p| \leq L + 32 \delta$.
We derive
\[
|x_0p| \simeq  |x_0 \xi^{-j}(x_0)| \pm (L + 48 \delta).
\]
Similarly,
\[
|x_0p_-| \simeq |x_0 \xi^{-j \mp 2}(x_0)| \pm (L + 48 \delta).
\]
Hence,
\begin{align}
|x_0p_-| & \simeq |x_0 \xi^{-j}(x_0)| \pm  | x_0 \xi^{-j \mp 2}(x_0)| \pm (L + 48 \delta) \\
 & \simeq |x_0 p| \pm (L + 48 \delta) \pm (2L + 32 \delta) \pm (L + 48 \delta) \\
 & \simeq |x_0 p| \pm (4L + 128 \delta). \label{eq:p0}
\end{align}

Now, from Lemma~\ref{lem:4d}, we have
\begin{align*}
|x_0 p| + |px| - 8 \delta & \leq |x_0 x| \leq  |x_0 p| + |px| \\
|x_0 p_-| + |p_-x_-| - 8 \delta & \leq |x_0 x_-| \leq  |x_0 p_-| + |p_- x_-|.
\end{align*}
By taking the difference, we obtain
\[
|x_0 p| - |x_0 p_-| + |px| - |p_px_-| - 8 \delta \leq |x_0 x| - |x_0 x_-| \leq  |x_0 p| - |x_0 p_-| + |px| - |p_-x_-| + 8 \delta.
\]
From the previous estimates~\eqref{eq:px} and~\eqref{eq:p0}, we deduce
\[
|x_0 x_-| \simeq |x_0 x| \pm (4L + 152 \delta).
\]
\end{proof}
\forgotten

\section{Geometric properties of twisted products}

In this section, we introduce the twisted product and show that the norm on~$G$ is almost multiplicative with respect to the twisted product.

\begin{definition}
The twisted product of two elements $\alpha, \beta$ in~$G$ is defined~as
\[
\alpha \star \beta = \left\{
\begin{array}{l}
\alpha \beta_+ \quad \mbox{ if } D_1 \mbox{ or } D_{-1} \mbox{ separates } \alpha^{-1}(\oo) \mbox{ and } \beta(\oo) \\
\alpha \beta_- \quad \mbox{ otherwise}
\end{array}
\right.
\]
Note that the twisted product~$\star$ is not associative. Furthermore, it has no unit and is not commutative.
\end{definition}

Recall that $||\alpha|| = |O \alpha(O)|$ for every $\alpha \in G$.

\begin{proposition} \label{prop:*}
Let $\alpha,\beta, \beta' \in G$. Then
\begin{enumerate}
\item By construction, $\alpha \star \beta=\alpha \beta_\varepsilon$ with $\varepsilon = \pm$ so that $D_{\pm 1}$ or $D_{\pm 2}$ separates $\alpha^{-1}(\oo)$ and~$\beta_\varepsilon(\oo)$. \label{separate} 
\item If $\alpha \star \beta = \alpha \star \beta'$ with $\beta$ and~$\beta'$ of the same sign, then $\beta = \beta'$. \label{same}
\item We have
\[
||\alpha \star \beta|| \simeq ||\alpha|| + ||\beta|| \pm \Delta_\star
\]
where $\Delta_\star = \Delta_\star(L,\delta,\epsilon) = 12L + 758 \delta +12 \epsilon \leq 20 L$. \label{2}
\end{enumerate}
\end{proposition}

\begin{proof}
For the first point.
If $\varepsilon=+$, then the result is clear by definition of the twisted product.
Namely, $D_{\pm 1}$ separates $\alpha^{-1}(\oo)$ and~$\beta_\varepsilon(\oo)$.
If $\varepsilon=-$, then $D_1$ and $D_{-1}$ do not separate $\alpha^{-1}(\oo)$ and~$\beta(\oo)$.
Assume that $\beta$ is positive, that is, $\beta(\oo) \in D_j$ with $j>0$ (the other case is similar).
This implies that $\alpha^{-1}(\oo)$ lies in a domain $D_i$ with $i \geq -1$, otherwise $D_{-1}$ would separate the two points.
Therefore, $D_{-2}$ separates $\alpha^{-1}(\oo)$ and~$\beta_-(\oo)$ since $\beta_-(\oo)$ lies in $D_{-j-2}$ from Proposition~\ref{prop:-}.

For the second point, if $\alpha \beta = \alpha \beta'$ or $\alpha (\beta)_- = \alpha (\beta')_-$ then $\beta = \beta'$ from Proposition~\ref{prop:-}.\eqref{-=}.
Switching $\beta$ and~$\beta'$ if necessary, we can assume that $\alpha \beta = \alpha (\beta')_-$, and so $\beta = (\beta')_-$.
From Proposition~\ref{prop:-}.\eqref{opp}, this implies that $\beta$ and~$\beta'$ have opposite signs, which is absurd.

For the last point.
We have
\[
|| \alpha \star \beta || = |\oo \, \alpha \beta_\varepsilon(\oo)| = | \alpha^{-1}(\oo) \beta_\varepsilon(\oo) |.
\]
Using the point~\eqref{separate}, we can apply Lemma~\ref{lem:+} to derive
\[
| \oo \alpha^{-1}(\oo) | + | \oo \beta_\varepsilon(\oo) | - 4L - 294 \delta - 4 \epsilon \leq || \alpha \star \beta || \leq  | \oo \alpha^{-1}(\oo) | + | \oo \beta_\varepsilon(\oo) |.
\]
That is,
\[
|| \alpha \star \beta || \simeq  || \alpha || + || \beta_\varepsilon || \pm (4L + 294 \delta + 4 \epsilon).
\]
Hence, by Proposition~\ref{prop:-}, we obtain
\[
|| \alpha \star \beta || \simeq || \alpha || + || \beta || \pm (12 L + 758 \delta + 12 \epsilon).
\]
\end{proof}

\section{Inserting hyperbolic elements}

The proposition established in this section will play an important role in the proof of the injectivity of the nonexpanding map defined in Section~\ref{sec:nonexp}.

\begin{lemma} \label{lem:decomp}
Let $\beta \in G$ and $\kappa \geq 4$.
\begin{enumerate}
\item 
\[
(\xi^{\kappa} \star \beta)_+ = \left\{
\begin{array}{l}
\xi^{\kappa} \beta_+ \mbox{ if } \beta \mbox{ is positive} \\
\xi^{\kappa} \beta_- \mbox{ if } \beta \mbox{ is negative}
\end{array}
\right.
\]
In particular, $(\xi^{\kappa} \star \beta)_+$ is positive. \label{pos} \\
\item 
\[
(\xi^{\kappa} \star \beta)_- = \left\{
\begin{array}{l}
\xi^{-\kappa} \beta_- \mbox{ if } \beta \mbox{ is positive} \\
\xi^{-\kappa-4} \beta_+ \mbox{ if } \beta \mbox{ is negative}
\end{array}
\right.
\]
In particular, $(\xi^{\kappa} \star \beta)_-$ is negative.
\end{enumerate}
Therefore, for every $\alpha \in G$, we have
\begin{equation} \label{eq:decomp}
\alpha\star(\xi^\kappa \star \beta) = \alpha \, \xi^{\kappa_*} \beta_{\varepsilon}
\end{equation}
with $\kappa_* = \pm \kappa \mbox{ or } -\kappa-4$, and $\varepsilon = \pm$.
\end{lemma}

\begin{proof}
The first point is clear since $\xi^{-\kappa}(\oo) \in D_{-\kappa}$.
We only check the second one.
Let $j = j(\beta)$ be the index of~$\beta$, \cf~Definition~\ref{def:sign}.

Suppose that $j>0$.
We observe that the image of~$O$ by~$\xi^{\kappa} \star \beta = \xi^\kappa \beta$ lies in~$D_{j+\kappa}$ and that the index of~$\xi^{\kappa} \star \beta$ is positive equal to~$j+\kappa$.
Thus, $(\xi^{\kappa} \star \beta)_- = \xi^{-2(j+\kappa)-1} \xi^\kappa \beta = \xi^{-\kappa} \beta_-$.

Suppose that $j<0$.
We observe that the image of~$O$ by~$\xi^{\kappa} \star \beta = \xi^{\kappa-2j+1} \beta$ lies in~$D_{-j+2+\kappa}$ and that the index of~$\xi^{\kappa} \star \beta$ is positive equal to~$-j+2+\kappa$.
Thus, $(\xi^{\kappa} \star \beta)_- = \xi^{-2(-j+2+\kappa)-1} \xi^{\kappa-2j+1} \beta = \xi^{-\kappa-4} \beta_-$.
\end{proof}

\begin{definition}
The \emph{stable norm} of the fixed hyperbolic element~$\xi$ is defined as the limit of the subadditive sequence $\frac{||\xi^k||}{k}$.
It is denoted by~$||\xi||_\infty$.
From~\cite[Proposition~6.4]{cdp}, the stable norm of~$\xi$ is positive and satisfies
\[
||\xi||_\infty \leq ||\xi|| =  L \leq ||\xi||_\infty + 16 \delta.
\]
Observe that $||\xi||_\infty \geq 0.9  L$ as $L \geq 300 \delta$.

An element $\alpha \in G$ is \emph{$\eta$-minimal modulo a normal subgroup}~\mbox{$N \lhd G$}, where $\eta \geq 0$, if for every $\alpha' \in \alpha N$,
\[
||\alpha|| \leq ||\alpha'|| + \eta.
\]

We denote by~$\langle \langle \xi \rangle \rangle$ the normal subgroup of~$G$ generated by~$\xi$.
\end{definition}

Recall that $d(\alpha,\beta) = |\alpha(\oo) \beta(\oo)|$ for every $\alpha,\beta \in G$.

\begin{proposition} \label{prop:key}
Let $\alpha_1, \alpha_2, \beta_1, \beta_2 \in G$ such that $\alpha_i$ is $\eta$-minimal modulo~$\langle \langle \xi \rangle \rangle$ for some $\eta \geq 0$.
Let $\kappa \geq 4$ be an integer such that $\kappa \geq (5 \Delta_\star + \frac{5}{2} \Delta_- + 40 \delta + \eta)/||\xi||_\infty$.
If $\alpha_1\star(\xi^\kappa \star \beta_1) = \alpha_2\star(\xi^\kappa \star \beta_2)$, then $d(\alpha_1,\alpha_2) \leq (\kappa+4) L + 4(\Delta_\star + \frac{1}{2} \Delta_- + 8 \delta)$.
\end{proposition}

\begin{remark} \label{rem:160}
When $\eta = \Delta_-$, we can take $\kappa = 160$.
\end{remark}

\begin{proof}
Let $\gamma=\alpha_i \star (\xi^\kappa \star \beta_i) = \alpha_i \xi^{\kappa_i} (\beta_i)_{\varepsilon_i}$ where $\kappa_i = \pm \kappa$ or $-\kappa-4$ and $\varepsilon_i = \pm$ according to the decomposition~\eqref{eq:decomp} of Lemma~\ref{lem:decomp}.
The segments from $\oo$ to~$A_i=\alpha_i(\oo)$, from $A_i$ to~$B_i=\alpha_i(\xi^{\kappa_i}(\oo))$, from $B_i$ to~$\gamma(\oo)$, and from~$\oo$ to~$\gamma(\oo)$ will be denoted by $a_i$, $c_i$, $b_i$ and~$c$ (these segments may not be unique, but we choose some).
We denote also by~$\bar{A}_i$ and~$\bar{B}_i$ the projections of~$A_i$ and~$B_i$ to~$c$ (which again may not be unique).
Observe that
\begin{equation} \label{eq:LL}
\kappa \, ||\xi||_\infty \leq |\kappa_i| \, ||\xi||_\infty \leq |A_i B_i| = ||\xi^{\kappa_i}|| \leq (\kappa+4) L.
\end{equation}
From Propositions~\ref{prop:-}.\eqref{beta-} and~\ref{prop:*}.\eqref{2}, we derive
\begin{align}
\length(a_i \cup c_i \cup b_i) & = ||\alpha_i|| + ||\xi^{\kappa_i}|| + ||(\beta_i)_{\varepsilon_i}|| \nonumber \\
& \simeq ||\alpha_i \star (\xi^\kappa \star \beta_i) || \pm (2 \Delta_\star + \Delta_-) \nonumber \\
& \simeq \length(c) \pm (2 \Delta_\star + \Delta_-). \label{eq:c}
\end{align}
Thus, by Lemma~\ref{lem:ell}, the arc $a_i \cup c_i \cup b_i$ lies at distance at most $\Delta_\star + \frac{1}{2} \Delta_- + 8 \delta$ from the geodesic~$c$ with the same endpoints.
In particular, we have
\begin{align}
|A_i \bar{A}_i| & \leq  \Delta_\star + \tfrac{1}{2} \Delta_- + 8 \delta \label{eq:Ai} \\
|B_i \bar{B}_i| & \leq \Delta_\star + \tfrac{1}{2} \Delta_- + 8 \delta \label{eq:Bi}
\end{align}

\noindent {\bf Fact 1.} \emph{The points $\oo$, $\bar{A}_i$ and~$\bar{B}_i$ are aligned in this order along~$c$.} \\

Indeed, from~\eqref{eq:c}, we have
\[
|\oo A_i| + |A_iB_i| + |B_i \gamma(\oo)| - 2 \Delta_\star - \Delta_- \leq |\oo \gamma(\oo)| \leq |\oo B_i| + |B_i \gamma(\oo)|.
\]
This implies that
\[
|\oo B_i| \simeq |\oo A_i| + |A_i B_i| \pm (2 \Delta_\star + \Delta_-)
\]
and so
\[
|\oo \bar{B}_i| \geq |\oo \bar{A}_i| + \kappa \, ||\xi||_\infty - (2 \Delta_\star + \Delta_-) - 2(\Delta_\star + \tfrac{1}{2} \Delta_- + 8 \delta)
\]
using~\eqref{eq:Ai}, \eqref{eq:Bi} and~\eqref{eq:LL}.
By our choice of~$\kappa$, we have $|\oo \bar{B}_i| \geq |\oo \bar{A}_i|$. \\

\noindent {\bf Fact 2.} \emph{The points $\oo$, $\bar{A}_1$ and~$\bar{B}_2$ are aligned in this order along~$c$. \\
Similarly, the points $\oo$, $\bar{A}_2$ and~$\bar{B}_1$ are aligned in this order along~$c$.} \\

Let us prove the first assertion.
Arguing by contradiction, we can assume from Fact~1 that the points $\oo$, $\bar{A}_2$, $\bar{B}_2$ and~$\bar{A}_1$ are aligned in this order along~$c$.
Combined with~\eqref{eq:Ai} and~\eqref{eq:Bi}, this implies that
\begin{align}
|\oo A_1| & \simeq |\oo \bar{A}_1| \pm (\Delta_\star + \tfrac{1}{2} \Delta_- + 8 \delta) \nonumber \\
& \simeq |\oo \bar{A}_2| + |\bar{A}_2 \bar{B}_2| + |\bar{B}_2 \bar{A}_1| \pm (\Delta_\star + \tfrac{1}{2} \Delta_- + 8 \delta) \nonumber \\
& \simeq |\oo A_2| + |A_2 B_2| + |B_2 A_1| \pm 5 (\Delta_\star + \tfrac{1}{2} \Delta_- + 8 \delta) \label{eq:OA}
\end{align}

Now, consider the projection 
\[
\pi:X \to X/\langle \langle \xi \rangle \rangle,
\]
where $X/\langle \langle \xi \rangle \rangle$ is endowed with the quotient distance, still denoted by~$| \cdot |$.

By definition, $A_i=\alpha_i(\oo)$ and $B_i = \alpha_i(\xi^{\kappa_i}(\oo)) = (\alpha_i \xi^{\kappa_i} \alpha_i^{-1})(\alpha_i(\oo))$.
Hence, $\pi(A_i) = \pi(B_i)$ since $\alpha_i \xi^{\kappa_i} \alpha_i^{-1} \in \langle \langle \xi \rangle \rangle$.

Note also that
\[
|\oo A_i| \simeq |\pi(\oo) \pi(A_i)| \pm \eta
\]
since $\alpha_i$ is $\eta$-minimal modulo~$\langle \langle \xi \rangle \rangle$.

Continuing with~\eqref{eq:OA} and since $\pi(A_2) = \pi(B_2)$, we obtain
\begin{align*}
|\oo A_1| & \geq |\pi(\oo) \pi(A_2)| + \kappa \, ||\xi||_\infty + |\pi(B_2) \pi(A_1)| - 5 (\Delta_\star + \tfrac{1}{2} \Delta_- + 8 \delta) \\
& \geq |\pi(\oo) \pi(A_1)| + \kappa \, ||\xi||_\infty - 5 (\Delta_\star + \tfrac{1}{2} \Delta_- + 8 \delta) \\
& \geq |\oo A_1| - \eta + \kappa \, ||\xi||_\infty - 5 (\Delta_\star + \tfrac{1}{2} \Delta_- + 8 \delta).
\end{align*}
Hence a contradiction from our choice of~$\kappa$. 
Similarly, we derive the second assertion of Fact~2. \\

From the order of the points~$\bar{A}_i$ and~$\bar{B}_i$ on~$c$ given by Facts~1 and~2, we deduce by~\eqref{eq:LL}, \eqref{eq:Ai} and~\eqref{eq:Bi} that
\[
| \bar{A}_1 \bar{A_2}| \leq \max \{ | \bar{A}_1 \bar{B}_1|, |\bar{A}_2 \bar{B}_2 | \} \leq (\kappa+4) L + 2(\Delta_\star + \tfrac{1}{2} \Delta_- + 8 \delta).
\]
Therefore,
\[
d(\alpha_1,\alpha_2) = |A_1 A_2| \leq (\kappa+4) L + 2(\Delta_\star + \tfrac{1}{2} \Delta_- + 8 \delta).
\]
\end{proof}

\section{A nonexpanding embedding} \label{sec:nonexp}

In this section, we finally establish the key proposition in the proof of the main theorem.
Namely, we construct a nonexpanding map~$\Phi:\bar{G}*\Z_2 \to G$ and show that it is injective in restriction to some separated subset~$\bar{G}_\rho * \Z_2$. \\

\forget
Let $G$ be a non-elementary group acting properly and cocompactly by isometries on a proper geodesic $\delta$-hyperbolic metric space~$X$.
Let $N$ be an infinite normal subgroup of~$G$.
The quotient group~$\bar{G} = G/N$, whose neutral element is denoted by~$\bar{e}$, acts by isometries on the quotient metric space~$\bar{X}=X/G$.
Since any infinite subgroup of a word hyperbolic group in the sense of Gromov contains a hyperbolic element, the subgroup~$N$ contains a hyperbolic isometry~$\xi$ whose minimal displacement can be assumed to be at least~$300 \delta$ by taking a large enough power of~$\xi$.
\forgotten

Let $G$ be a finitely generated group acting by isometries on a proper geodesic $\delta$-hyperbolic metric space~$X$.
Let $N$ be a normal subgroup of~$G$.
The quotient group~$\bar{G} = G/N$, whose neutral element is denoted by~$\bar{e}$, acts by isometries on the quotient metric space~$\bar{X}=X/G$.

Suppose that~$N$ contains a hyperbolic isometry~$\xi$ of~$X$, \cf~Example~\ref{ex:hyp}.
By taking a large enough power of~$\xi$ if necessary, we can assume that the minimal displacement~$L$ of~$\xi$ is at least~$300 \delta$.

Given $\epsilon \in (0,\delta)$, let $O \in X$ be an origin with~$|O\xi(O)| \leq L + \epsilon$.
Denote by~$|| \cdot ||$ and~$|| \cdot ||_{\bar{G}}$ the norms induced on~$G$ and~$\bar{G}$ by the distance~$d$, \cf~\eqref{eq:dg}, and the quotient distance~$\bar{d}$.
We also need to fix $\kappa=160$, \cf~Remark~\ref{rem:160}. 

For every~$\gamma \in \bar{G}$, we fix once and for all a representative~$\alpha$ in~$G$ which is $\nu$-minimal modulo~$N$.
(To avoid burdening the arguments by epsilontics, we will actually assume that $\nu =0$.)
This yields an embedding 
\begin{equation*} 
\Phi : \bar{G} \hookrightarrow G.
\end{equation*}
We extend this embedding to a map
\[
\Phi:\bar{G} * \Z_2 \to G
\]
by induction on~$m$ with the relation
\[
\Phi(\gamma_1 * 1 * \cdots * \gamma_{m+1}) = \Phi(\gamma_1)_\varepsilon \star (\xi^\kappa \star \Phi(\gamma_2 * 1 * \cdots * \gamma_{m+1}))
\]
where $\gamma_1 * 1 * \cdots * \gamma_{m+1}$ is in reduced form and $\varepsilon$ is the sign of~$ \Phi(\gamma_2 * 1 * \cdots * \gamma_{m+1})$, unless $\gamma_1=\bar{e}$, in which case $\varepsilon = +$.

For $\lambda \geq 0$, consider the norm~$||\cdot||_\lambda$ on the group~$\bar{G} * \Z_2$ defined as
\[
|| \gamma_1 * 1 * \cdots * \gamma_{m+1} ||_\lambda = \sum_{i=1}^{m+1} ||\gamma_i ||_{\bar{G}} + m \lambda.
\]
Let $\alpha_i = \Phi(\gamma_i)$.
From Propositions~\ref{prop:-}.\eqref{beta-} and~\ref{prop:*}.\eqref{2}, we derive
\begin{align*}
|| \Phi(\gamma_1 * 1 * \cdots * \gamma_{m+1}) || & \leq || \alpha_1|| + || \xi^\kappa || + || \Phi(\gamma_2 * 1 * \cdots * \gamma_{m+1}) || + 2\Delta_\star + \Delta_- \\
& \leq \sum_{i=1}^{m+1} || \alpha_i || + m (\kappa L + 2 \Delta_\star + \Delta_-)
\end{align*}
Thus, since $||\alpha_i|| = ||\gamma_i ||_{\bar{G}}$, the map~$\Phi$ does not increase the ``norms'' if $\lambda \geq \kappa L + 2 \Delta_\star + \Delta_-$.

For~$\rho >0$, consider a maximal system of disjoint (closed) balls of radius~$\rho/2$ in~$\bar{G}$ containing the ball of radius~$\rho/2$ centered at~$\bar{e}$.
By maximality, the centers of these balls form a subset~$\bar{G}_\rho$ of~$\bar{G}$ such that
\begin{enumerate}
\item the elements of~$\bar{G}_\rho$ are at distance greater than~$\rho$ from each other; 
\item every element of~$\bar{G}$ is at distance at most~$\rho$ from an element of~$\bar{G}_\rho$.
\end{enumerate}
Note that $||\gamma||_{\bar{G}} > \rho$ for every $\gamma \in \bar{G}_{\rho}$ with $\gamma \neq \bar{e}$.

Consider the subset~$\bar{G}_{\rho} * \Z_2$ of~$\bar{G}*\Z_2$  formed of the elements
\[
\gamma_1 * 1 * \cdots * \gamma_{m+1}
\]
where $m \in \N$ and~$\gamma_i \in \bar{G}_\rho$.
Note that $\bar{G}_\rho * \Z_2$ is not a group, as $\bar{G}_\rho$ itself is not a group.

\begin{proposition} \label{prop:inj}
Let  $\lambda \geq \kappa L + 2 \Delta_\star + \Delta_-$ and $\rho \geq (\kappa+4) L + 4(\Delta_\star +\tfrac{1}{2} \Delta_- + 8 \delta)$.
Then the nonexpanding map $\Phi:(\bar{G}_\rho * \Z_2, || \cdot ||_\lambda) \to (G, || \cdot ||)$ is injective.
\end{proposition}

\begin{remark}
We can take $\lambda = 210 L$ and $\rho = 270 L$.
\end{remark}

\begin{proof}
Assume 
\begin{equation} \label{eq:assume}
\Phi(\gamma_1 * 1 * \cdots * \gamma_{m+1}) = \Phi(\gamma'_1 * 1 * \cdots * \gamma'_{m'+1}).
\end{equation}
Let $\alpha_1=\Phi(\gamma_1)$ and $\varepsilon$ be the sign of $ \Phi(\gamma_2 * 1 * \cdots * \gamma_{m+1})$,unless $\gamma_1 = \bar{e}_1$ in which case~$\varepsilon = +$.
Similar definitions hold for~$\alpha'_1$ and~$\varepsilon'$ by replacing~$\gamma_i$ with~$\gamma'_i$.

By Proposition~\ref{prop:-}.\eqref{beta-}, for every $\alpha \in G$ which is minimal modulo~$N$, and so modulo~$\langle \langle \xi \rangle \rangle$, its symmetric~$\alpha_-$ is $\eta$-minimal modulo~$\langle \langle \xi \rangle \rangle$ with $\eta=\Delta_-$.
Thus, from Proposition~\ref{prop:key}, we have
\[
d((\alpha_1)_\varepsilon,(\alpha'_1)_{\varepsilon'}) \leq \rho.
\]
By definition of the quotient distance,
\[
\bar{d}(\gamma_1,\gamma'_1) \leq d(\xi^i \alpha_1, \xi^{i'} \alpha'_1)
\]
for every $i,i' \in \Z$.
Hence,
\[
\bar{d}(\gamma_1,\gamma'_1) \leq d((\alpha_1)_\varepsilon, (\alpha'_1)_{\varepsilon'}) \leq \rho.
\]
Since $\gamma_1$ and~$\gamma'_1$ lie in a $\rho$-separated set, this shows that $\gamma_1 = \gamma'_1$ and so $\alpha_1 = \alpha'_1$.
 
Suppose $\varepsilon \neq \varepsilon'$.
In this case, $\alpha=\alpha_1$ is distinct from~$\bar{e}$ and $d(\alpha,\alpha_-) \leq \rho$.
Now, as $\alpha^{-1} \star \alpha = \alpha^{-1} \alpha_-$, we derive from Proposition~\ref{prop:*}.\eqref{2} that
\[
d(\alpha,\alpha_-) = || \alpha^{-1} \alpha_- || = || \alpha^{-1} \star \alpha|| \geq 2 ||\alpha|| - \Delta_\star \geq 2 \rho - \Delta_\star > \rho.
\]
From this contradiction, we deduce that $\varepsilon = \varepsilon'$.

Substituting this into~\eqref{eq:assume}, we derive
\[
\alpha_\varepsilon \star ( \xi^\kappa \star \beta) = \alpha_\varepsilon \star (\xi^\kappa \star \beta')
\]
where $\beta =  \Phi(\gamma_2 * 1 * \cdots * \gamma_{m+1})$ and with a similar definition for~$\beta'$ by replacing~$\gamma_i$ with~$\gamma'_i$.
As both $\xi^\kappa \star \beta$ and $\xi^\kappa \star \beta'$ are positive, \cf~Lemma~\ref{lem:decomp}.\eqref{pos}, we obtain from Proposition~\ref{prop:*}.\eqref{same} that $\xi^\kappa \star \beta = \xi^\kappa \star \beta'$.
Since $\beta$ and~$\beta'$ are of the same sign (\ie, $\varepsilon = \varepsilon'$), we similarly derive that $\beta = \beta'$.

We conclude by induction that $m = m'$ and $\gamma_i=\gamma'_i$.
\end{proof}

\section{Conclusion}

Consider a non-elementary group~$G$ acting properly and cocompactly by isometries on a proper geodesic $\delta$-hyperbolic metric space with fixed origin~$O$.
Denote by~$\Delta$ the diameter of~$X/G$.
Let $N$ be an infinite normal subgroup of~$G$ and $\bar{G}=G/N$.
Denote by~$L$ the maximal value between $300 \delta$ and the minimal norm of a hyperbolic element in~$N$.

Combining Propositions~\ref{prop:strict}, \ref{prop:rho} and~\ref{prop:inj}, we obtain
\[
\omega(G) \geq \omega(\bar{G}_{\rho} * \Z_2 , \lambda) \geq \omega(\bar{G}*\Z_2,\tilde{\lambda}) \geq \omega(\bar{G}) + \frac{1}{4 \tilde{\lambda}} \log(1+e^{- \tilde{\lambda} \omega(\bar{G})})
\]
where $\lambda = 210 L$, $\rho=270L$ and~$\tilde{\lambda} = 2\lambda + 15(\Delta+\rho) \, \card B_{\bar{G}}(3(\Delta+\rho))$.
One could also replace $\card B_{\bar{G}}(3(\Delta+\rho))$ with $\card B_{G}(3(\Delta+\rho))$, since the latter cardinal is at least as large as the former.

Observe that if $\omega(\bar{G})$ is close to~$\omega(G)$, then $L$ is large, that is, the norm of every hyperbolic element of~$N$ is large.

\forget
\section{Exponential growth rate}

Let $\Gamma$ be a discrete group endowed with a (left-invariant) norm~$|| \cdot ||$ taking discrete values.
For~$\rho >0$, consider a maximal system~$\mathcal{S}$ of disjoint balls of radius~$\rho/2$ in~$\Gamma$.
The centers of these balls form a $\rho$-separated subset of~$\Gamma$ denoted by~$\Gamma_\rho$.
Let $B_\Gamma(R)$ be the ball of radius~$R$ formed of the elements~$\gamma$ in~$\Gamma$ of norm at most~$R$.
Define the \emph{exponential growth rate} of~$\Gamma$ as
\[
\omega(\Gamma) = \limsup_{R \to \infty} \frac{\log \card B_\Gamma(R)}{R}.
\]
Similarly, we define~$B_{\Gamma_\rho}(R)$ and~$\omega(\Gamma_\rho)$ replacing~$\Gamma$ with~$\Gamma_\rho$.

\begin{proposition}
\mbox{ }
\begin{enumerate}
\item We have 
\[
\omega(\Gamma_\rho) = \omega(\Gamma).
\]
\item Suppose that there exists $A,\omega >0$ and $R_0 \geq 0$ such that $V_\Lambda(R) \geq A e^{\omega R}$ for every $R \geq R_0$.
Then
\[
\omega(\Gamma_\rho * \Z^2) \geq \omega + \frac{1}{R_0 + d} \log(1+A e^{- \omega d}).
\]
\end{enumerate}
\end{proposition}

\begin{proof}
For the first point, by maximality of the system~$\mathcal{S}$ of disjoint balls in the definition of~$\Gamma_\rho$, every point of~$\Gamma$ is at distance at most~$\rho$ from~$\Gamma_\rho$.
Thus,
\[
B_{\Gamma_\rho}(R) \subset B_\Gamma(R) \subset \bigcup_{x \in B_{\Gamma_\rho}(R+\rho)} B_\Gamma(\rho),
\]
which yields the bounds
\[
\card B_{\Gamma_\rho}(R) \leq \card B_\Gamma(R) \leq \card B_{\Gamma_\rho}(R+\rho) \, \card B_\Gamma(\rho).
\]
The relation $\omega(\Gamma_\rho) = \omega(\Gamma)$ follows.

For the second point, let $k$ be a positive integer and $R=R_0+d$.
Every point of

\end{proof}

We have 
\[
\omega(\Gamma_\rho) = \omega(\Gamma) =:\omega.
\]
Furthermore, there exists $A>0$ such that $\card B_{\Gamma_\rho}(R) \geq A \, e^{\omega R}$ for every $R \geq 0$.
In this case,
\[
\omega(\Gamma_\rho * \Z_2) \geq \omega(\Gamma_\rho) + ...
\]

\begin{proof}
Slightly perturbing the norm~$||\cdot||$ if necessary, we can assume that~$|| \cdot||$ takes different values for different elements of~$\Gamma_\rho$, that is, $|| \cdot ||$ is injective on~$\Gamma_\rho$.
Here, we do not require~$||\cdot||$ to be a norm anymore.

Define~$D_{\Gamma_\rho}(R)$ as the ball~$B_{\Gamma_\rho}(\tau)$ of radius
\[
\tau = \inf \{ r >0 \mid \card B_{\Gamma_\rho}(r) \geq A \, e^{\omega R} \}.
\]
Note that $\tau \leq R$, that is, $D_{\Gamma_\rho}(R) \subset B_{\Gamma_\rho}(R)$.
As~$|| \cdot ||$ takes discrete values and is supposed to be injective on~$\Gamma_\rho$, we have
\[
A \, e^{\omega R} \leq \card D_{\Gamma_\rho}(R) < A \, e^{\omega R} +1.
\]
For every $p \in \N^*$ and $R >0$, define the annulus
\[
\mathcal{A}_p = D_{\Gamma_\rho}(pR-d) \setminus D_{\Gamma_\rho}((p-1)R-d).
\]
We have
\begin{align*}
\card \mathcal{A}_p & \geq A \, e^{\omega (pR-d)} - A \, e^{\omega ((p-1)R-d)} - 1 \\
& \geq C \, e^{\omega (pR-d)}
\end{align*}
where $C = A \, \max \{ 1 - A e^{-\omega R} - \frac{1}{A} e^{-\omega (R-d)} , 0\}$.

Let $k \in \N^*$ and $k_1,\cdots,k_{m+1} \in \N^*$ with $\sum_{i=1}^{m+1} k_i = k$.
As $D_{\Gamma_\rho}(k_i R)$ is contained in $B_{\Gamma_\rho}(k_i R)$, every element $\gamma = \gamma_1 * 1 * \cdots * \gamma_{m+1}$ of~$\Gamma_\rho$ with $\gamma_i \in \mathcal{A}_{k_i}$ satisfies
\[
|| \gamma||_d \leq \sum_{i=1}^{m+1} (k_i R -d) + md \leq kR.
\]
Hence,
\[
\bigcup_{m=0}^{k-1} \, \bigcup_{\sum k_i = k} \mathcal{A}_{k_1} * 1 * \cdots * \mathcal{A}_{k_{m+1}} \subset B_{\Gamma_\rho * \Z_2} (kR)
\]
where the second union is taken over all integers $k_1,\cdots,k_{m+1} \geq 1$ such that $\sum_{i=1}^{m+1} k_i = k$.
Furthermore, this union is disjoint since the annuli $\mathcal{A}_k$ are disjoint.
Therefore, 
\begin{align*}
\card B_{\Gamma_\rho * \Z_2} (kR) & \geq \sum_{m=0}^{k-1} \, \sum_{\sum k_i = k} C^{m+1} e^{\omega (kR-(m+1)d)} \\
& \geq C \, e^{\omega kR} \, e^{-\omega d}  \sum_{m=0}^{k-1} \, \sum_{\sum k_i = k} C^{m} e^{- m \omega d} \\
& \geq C \, e^{\omega kR} \, e^{-\omega d}  \sum_{m=0}^{k-1} { k-1 \choose m } C^{m} e^{- m \omega d} \\
& \geq C \, e^{\omega kR} \, e^{-\omega d}  \left( 1+C e^{-\omega d} \right)^{k-1}.
\end{align*}
By taking the $\log$, dividing by~$kR$ and passing to the limit when $k$ goes to infinity, we derive
\[
\omega(\Gamma_\rho * \Z_2) \geq \omega(\Gamma) + \frac{1}{R} \log(1+ C e^{-\omega d}) > \omega(\Gamma).
\]
\end{proof}

\section{misc}

Let $A$ be a countable set and~$|| \cdot ||$ be a nonnegative function on~$A$ such that $B_A(R) = \{ \alpha \in A \mid || \alpha || \leq R \}$ is finite for~$R \geq 0$.
In practice, $A$ will be a subset of a metric group~$G$ and $||\alpha||$ will represent the distance from the neutral element of~$G$ to~$\alpha$.

The Poincar\'e series of~$A=(A,|| \cdot ||)$ is defined as the formal series
\[
\mathcal{P}_A(s) = \sum_{\alpha \in A} e^{-s ||\alpha||}.
\]

The critical exponent of~$A=(A,|| \cdot ||)$, denoted by~$\delta(A)$, is defined as the abscissa of convergence of~$\mathcal{P}_A$.
In other words, it is the unique element of~$[0,\infty]$ such that $\mathcal{P}_A(s)$ converges if $s>\delta(A)$ and diverges if~$s<\delta(A)$.

The exponential growth rate of~$A=(A,||\cdot||)$, denoted by~$\omega(A)$, is defined as in~\eqref{eq:om} by replacing~$G$ with~$A$ and $R$ with~$k \in \Z$ in the right-hand side of the equation.

\begin{lemma} \label{lem:critical}
Let $A=(A, || \cdot||)$ be as above.
Then $\delta(A) = \omega(A)$.
\end{lemma}

\begin{proof}
The Poincar\'e series~$\mathcal{P}_A(s) = \sum_{\alpha \in A} e^{-s ||\alpha||}$ is of the same nature as the series
$
\sum_{k \in \N} a_k \, e^{-sk}
$
with~$a_k = \card B_A(k) - \card B_A(k-1)$, since the two series differ at most by a bounded factor.
Thus, the Poincar\'e series~$\mathcal{P}_A(s)$ is of the same nature as
$
\sum_{k \in \N} b_k \, e^{-sk}
$,
where $b_k=a_0+a_1+ \cdots + a_k = \card B_A(k)$, from Abel's summation formula.
Since the series~$\sum_{k \in \N} b_k \, e^{-sk}$ converges for~$s > \omega(A)$ and diverges for~$s < \omega(A)$, the result follows.
\end{proof}

\begin{proposition} \label{prop:rho}
For $\displaystyle \lambda' \geq \rho + \frac{\log \card B_{\bar{G}}(\rho)}{\omega(\bar{G})}$, we have
\[
\omega(\bar{G} * \Z_2,\lambda+\lambda') \leq \omega(\bar{G}_\rho * \Z_2,\lambda).
\] 
\end{proposition}

\begin{proof}
The Poincar\'e series~$\mathcal{P}_{\bar{G}*\Z_2,\lambda}(s)$ is of the same nature as
\[
\sum_{m=0}^\infty \sum_{\gamma_1, \cdots, \gamma_{m+1} \in \bar{G}} e^{-s ||\gamma_1 * 1* \cdots * \gamma_{m+1}  ||_\lambda} = \sum_{m=0}^\infty e^{-s m \lambda} \sum_{\gamma_1, \cdots, \gamma_{m+1} \in \bar{G}} e^{-s ||\gamma_1||_{\bar{G}}} \cdots e^{-s ||\gamma_{m+1}||_{\bar{G}}}.
\]
Thus, the Poincar\'e series~$\mathcal{P}_{\bar{G}*\Z_2,\lambda}(s)$ is of the same nature as
\begin{equation} \label{eq:P}
\sum_{m=0}^\infty \left( e^{-s \lambda} \sum_{\gamma \in \bar{G}} e^{-s ||\gamma ||_{\bar{G}}} \right)^{m+1}.
\end{equation}
We obtain a similar result for~$\mathcal{P}_{\bar{G}_\rho*\Z_2,\lambda}(s)$ by replacing~$\bar{G}$ with~$\bar{G}_\rho$ in~\eqref{eq:P}.

On the other hand, the Poincar\'e series of~$\bar{G}$ satisfies

\begin{align*}
\mathcal{P}_{\bar{G}}(s) = \sum_{\gamma \in \bar{G}_\rho} e^{-s || \gamma ||_{\bar{G}}} & \leq \sum_{\mathring{\gamma} \in \bar{G}_\rho} \sum_{\gamma \in B_{\bar{G}}(\mathring{\gamma},\rho)} e^{-s ||\gamma||_{\bar{G}}} \\
& \leq \sum_{\gamma \in \bar{G}_\rho} \card B_{\bar{G}}(\rho) \, e^{-s(||\mathring{\gamma}||_{\bar{G}} - \rho)} \\
& \leq \card B_{\bar{G}}(\rho) \, e^{s \rho} \, \sum_{\gamma \in \bar{G}_\rho} e^{-s ||\mathring{\gamma}||_{\bar{G}}}
\end{align*}

Replacing~$\lambda$ with~$\lambda+\lambda'$ in~\eqref{eq:P}, we derive that the Poincar\'e series~$\mathcal{P}_{\bar{G}*\Z_2,\lambda+\lambda'}(s)$ is of the same nature as
\begin{equation} \label{eq:m+1}
\sum_{m=0}^\infty \left( \card B_{\bar{G}}(\rho) \, e^{-s(\lambda' - \rho)} \right)^{m+1} \, \left( e^{-s \lambda} \sum_{\mathring{\gamma} \in \bar{G}_\rho} e^{-s ||\mathring{\gamma} ||_{\bar{G}}} \right)^{m+1}.
\end{equation}

For $s< \delta(\bar{G} * \Z_2, \lambda + \lambda')$, our choice of~$\lambda'$ guarantees that 
\mbox{$\card B_G(\rho) \, e^{-s(\lambda' - \rho)} \leq 1$.}
In this case, the divergent series~\eqref{eq:m+1} is bounded from above by
\[
\sum_{m=0}^\infty \left( e^{-s \lambda} \sum_{\mathring{\gamma} \in \bar{G}_\rho} e^{-s ||\mathring{\gamma} ||_{\bar{G}}} \right)^{m+1}
\]
which is of the same nature as~$\mathcal{P}_{\bar{G}_\rho * \Z_2,\lambda}(s)$ from the remark following~\eqref{eq:P}.
Thus, the Poincar\'e series~$\mathcal{P}_{\bar{G}_\rho*\Z_2,\lambda}(s)$ is also divergent for $s<  \delta(\bar{G} * \Z_2, \lambda + \lambda')$.
Therefore, $ \delta(\bar{G}_\rho * \Z_2, \lambda) \geq  \delta(\bar{G} * \Z_2, \lambda + \lambda')$.
The proposition follows from Lemma~\ref{lem:critical}.
\end{proof}


The following simple fact holds.

\begin{lemma}
We have 
\[
\omega(\bar{G}_\rho) = \omega(\bar{G}).
\]
\end{lemma}

\begin{proof}
By maximality of the system~$\mathcal{S}$ of disjoint balls in the definition of~$\bar{G}_\rho$, every point of~$\bar{G}$ is at distance at most~$\rho$ from~$\bar{G}_\rho$.
Thus,
\[
B_{\bar{G}_\rho}(R) \subset B_{\bar{G}}(R) \subset \bigcup_{x \in B_{\bar{G}_\rho}(R+\rho)} B_{\bar{G}}(\rho),
\]
which yields the bounds
\[
\card B_{\bar{G}_\rho}(R) \leq \card B_{\bar{G}}(R) \leq \card B_{\bar{G}_\rho}(R+\rho) \cdot \card B_{\bar{G}}(\rho).
\]
The relation $\omega(\bar{G}_\rho) = \omega(\bar{G})$ follows.
\end{proof}

\forgotten

\end{document}